\documentclass{amsart}
\usepackage{amsfonts}
\usepackage{amsmath}
\usepackage{amssymb}
\usepackage{graphicx}

\theoremstyle{plain}
\newtheorem{theorem}{Theorem}[section]

\newtheorem{proposition}{Proposition}[section]

\newtheorem{lemma}{Lemma}[section]

\theoremstyle{remark}
\newtheorem{remark}{Remark}[section]

\numberwithin{equation}{section}
\newcounter{saveenumi}
\input{tcilatex}

\begin{document}
\title[Locating the peaks of least-energy solutions]{Locating the peaks of
least-energy solutions to a quasilinear elliptic Neumann problem}
\author{Yi Li}
\address{Department of Mathematics\\
The University of Iowa\\
Iowa City, IA 52242}
\address{Department of Mathematics\\
Hunan Normal University\\
Changsha, Hunan}
\email{yi-li@uiowa.edu}
\author{Chunshan Zhao}
\address{Department of Mathematical Sciences\\
Georgia Southern University\\
Statesboro, GA 30460}
\email{czhao@GeorgiaSouthern.edu}

\begin{abstract}
In this paper we study the shape of least-energy solutions to the
quasilinear problem $\varepsilon^{m}\Delta_{m}u-u^{m-1}+f\left( u\right) =0$
with homogeneous Neumann boundary condition. We use an intrinsic variation
method to show that as $\varepsilon\rightarrow0^{+}$, the global maximum
point $P_{\varepsilon}$ of least-energy solutions goes to a point on the
boundary $\partial\Omega$ at the rate of $o(\varepsilon)$ and this point on
the boundary approaches to a point where the mean curvature of $%
\partial\Omega$ achieves its maximum. We also give a complete proof of
exponential decay of least-energy solutions.
\end{abstract}

\maketitle

\section{\label{Int}Introduction and statement of results}

In this paper we study the shape of certain solutions to the following
quasilinear elliptic Neumann problem:%
\begin{equation}
\left\{ \begin{aligned} &\varepsilon^{m}\Delta_{m}u-u^{m-1}+f\left( u\right)
=0, & & u>0\text{ in }\Omega, \\ &\frac{\partial u}{\partial\nu}=0 & &
\text{on }\partial\Omega, \end{aligned}\right.  \label{eqInt.1}
\end{equation}
where $m$ ($2\leq m<N$) and $0<\varepsilon\leq 1$ are constants and $\Omega
\subseteq\mathbb{R}^{N}$ ($N\geq3$) is a smooth bounded domain. The operator 
$\Delta_{m}u=\limfunc{div}(\left\vert \nabla u\right\vert ^{m-2}\nabla u)$
is the $m$-Laplacian operator, and $\nu$ is the unit outer normal to $%
\partial\Omega$.

Problem (\ref{eqInt.1}) appears in the study of non-Newtonian fluids,
chemotaxis and biological pattern formation. For example, in the study of
non-Newtonian fluids, the quantity $m$ is a characteristic of the medium:
media with $m>2$ are called dilatant fluids, and those with $m<2$ are called
pseudo-plastics. If $m=2$, they are Newtonian fluids (see \cite{Dia85} and
its bibliography for more backgrounds). For the case $m=2$, (\ref{eqInt.1})
is also known as the stationary equation of the Keller--Segal system in
chemotaxis \cite{LNT88} or the limiting stationary equation of the so-called
Gierer--Meinhardt system in biological pattern formation (see \cite{Wei97}).

First let us recollect some results related to our problem. In a series of
remarkable papers, C.-S. Lin, W.-M. Ni and I. Takagi \cite{LNT88}, Ni and
Takagi \cite{NiTa91}, \cite{NiTa93} studied the Neumann problem for certain
elliptic equations, including%
\begin{equation}
\left\{ \begin{aligned} &d\Delta u-u+u^{p}=0, & &u>0\text{ in }\Omega,\\
&\frac{\partial u}{\partial\nu}=0 & &\text{on }\partial\Omega, \end{aligned}%
\right.  \label{eqInt.2}
\end{equation}
where $d>0$, $p>1$ are constants, and $p$ is subcritical, i.e., $p<\frac {N+2%
}{N-2}$. First, Lin, Ni and Takagi \cite{LNT88} applied the mountain-pass
lemma \cite{AmRa73} to show the existence of a least-energy solution $u_{d}$
to (\ref{eqInt.2}), by which is meant that $u_{d}$ has the least energy
among all solutions to (\ref{eqInt.2}) with the energy functional%
\begin{equation*}
I_{d}\left( u\right) =\int_{\Omega}\left( \frac{d}{2}\left\vert \nabla
u\right\vert ^{2}+\frac{1}{2}u^{2}-\frac{1}{p+1}u_{+}^{p+1}\right) \,dx
\end{equation*}
defined on $W^{1,2}\left( \Omega\right) $. Hereinafter $u_{+}=\max\left\{
u,0\right\} $ and $u_{-}=\min\left\{ u,0\right\} $. Then in \cite{NiTa91}, 
\cite{NiTa93}, Ni and Takagi investigated the shape of the least-energy
solution $u_{d}$ as $d$ becomes sufficiently small, and showed that $u_{d}$
has exactly one peak (i.e., local maximum of $u_{d}$) at $P_{d}\in
\partial\Omega$. Moreover, as $d$ tends to zero, $P_{d}$ approaches a point
where the mean curvature of $\partial\Omega$ achieves its maximum. See \cite%
{N} for a review in this field. Also see \cite{NPT92} for the critical case $%
p=\frac{N+2}{N-2}$, and \cite{Gui96}, \cite{GuGh98}, \cite{GuWe99}, \cite%
{GuWe00}, \cite{GWW00} for existence and properties of multiple-peaks
solutions to (\ref{eqInt.2}).

From now on we make some hypotheses on $f\colon\mathbb{R}\rightarrow \mathbb{%
R}$, as follows.

\begin{enumerate}
\item \renewcommand{\theenumi}{H$_\arabic{enumi}$}

\item \label{hypothesis1}$f\left( t\right) \equiv0$ for $t\leq0$ and $f\in
C^{1}\left( \mathbb{R}\right) $.

\item \label{hypothesis2}$f(t)=O\left( t^{p}\right) $ as $t\rightarrow
\infty $ with $m-1<p<\displaystyle\frac{N\left( m-1\right) +m}{N-m}$.

\item \label{hypothesis3}Let $\displaystyle F\left( t\right)
=\int_{0}^{t}f\left( s\right) \,ds$. Then there exists a constant $%
\displaystyle
\theta\in\left( 0,\frac1m\right) $ such that $F\left( t\right) \leq\theta
tf\left( t\right) $ for $t>0$.

\item \label{hypothesis4}$\displaystyle\frac{f\left( t\right) }{t^{m-1}}$ is
strictly increasing for $t>0$ and $f\left( t\right) =O\left( t^{m-1+\delta
}\right) $ as $t\rightarrow0^{+}$ with a constant $\delta>0$.

\item \label{hypothesis5}Let $\displaystyle g\left( u\right) =\frac{\left(
m-1\right) u^{m-1}-uf^{\prime}\left( u\right) }{u^{m-1}-f\left( u\right) }$.
Then $g\left( u\right) $ is non-increasing on $\left( u_{c},\infty\right) $,
where $u_{c}$ is the unique positive solution for $f\left( t\right) =t^{m-1}$%
.
\end{enumerate}

\noindent\setcounter{saveenumi}{\value{enumi}} Next we present some
preliminary knowledge about least energy solutions of the following problem: 
\begin{equation}
\left\{ 
\begin{array}{l}
\Delta_{m}u-u^{m-1}+f(u)=0\quad\text{in }\mathbb{R}^{N} \\[2pt] 
u>0\quad\text{in }\mathbb{R}^{N}%
\end{array}
\right.  \label{IntroEq5}
\end{equation}
As before we define an ``energy functional'' $I$:$W^{1,m}(\mathbb{R}%
^N)\longrightarrow\mathbb{R}$ associated with (\ref{IntroEq5}) by 
\begin{equation}
I(\tilde v)=\int_{\mathbb{R}^N}\left\{ \frac{1}{m}\left( \varepsilon
^{m}\left\vert \nabla \tilde v\right\vert ^{m}+\left\vert \tilde
v\right\vert ^{m}\right) -F(\tilde v_{+})\right\} \;dx\text{.}
\label{IntroEq003}
\end{equation}

Next let us give a remark on ground states to the problem \ref{IntroEq5}.
Here by a ground state we mean a non-negative nontrivial $C^1$ distribution
solution which tends to zero at $\infty$. For case $m=2$, it is well known
that the problem \ref{IntroEq5} has a unique ground state (up to
translations) which is radially symmetric \cite{GNN81}. For case $2<m<N$
uniqueness and radial symmetry of ground states are still open. But the
Steiner symmetrization tells us the least-energy solutions must be radially
symmetric (certainly least-energy solutions are ground states). Our
assumptions guarantee that the uniqueness (up to translations) of radial
ground states (see \cite{SeTa00}), which implies the uniqueness of
least-energy solutions of the problem (\ref{IntroEq5}). Exact exponential
decay of radial ground states was given in \cite{LiZh04a}, thus we have the
following proposition about the unique radial least-energy solution to
problem \ref{IntroEq5}:

\begin{proposition}
\label{IntroProp1}Under assumptions \textup{(\ref{hypothesis1})--(\ref%
{hypothesis5})}, there is a unique least energy solution $w(x)$ for \emph{(}%
\ref{IntroEq5}\emph{)} satisfying:\medskip

\noindent\emph{(}\thinspace i\thinspace\emph{)} \setlength{\hangindent}{18pt}%
$w$ is radial, i.e., $w(x)=w(\left\vert x\right\vert )=w(r)$ and $w\in C^{1}(%
\mathbb{R}^{N})$ with \newline
$w(0)=\max_{X\in\mathbb{R}^{N}}w(x)$, $w^{\prime}(0)=0$ and $w^{\prime
}(r)<0,\;\forall r>0$.

\noindent\setlength{\hangindent}{18pt}\emph{(}ii\emph{)} $\lim
_{r\longrightarrow\infty}w(r)r^{\frac{N-1}{m(m-1)}}e^{\binom{1}{m-1}^{\frac{1%
}{m}}r}=C_{0}>0$ for some constant $C_{0}$ and

$\lim _{r\longrightarrow\infty}\frac{w^{\prime}(r)}{w(r)}=-\binom{1}{m-1}^{%
\frac{1}{m}}$. \medskip
\end{proposition}

\begin{remark}
\label{RemInt.1}A good example for $f\left( t\right) $ which satisfies all
hypotheses \textup{(\ref{hypothesis1})--(\ref{hypothesis5})} is $f\left(
t\right) =t^{p}$ for $\displaystyle m-1<p<\frac{N\left( m-1\right)+m}{N-m}$.
\end{remark}

Next we define an \textquotedblleft energy functional\textquotedblright\ $%
J_{\varepsilon}\colon W^{1,m}\left( \Omega\right) \rightarrow\mathbb{R}$
associated with (\ref{eqInt.1}) by%
\begin{equation}
J_{\varepsilon}\left( v\right) =\int_{\Omega}\left\{ \frac{1}{m}\left(
\varepsilon^{m}\left\vert \nabla v\right\vert ^{m}+\left\vert v\right\vert
^{m}\right) -F\left( v_+\right) \right\} \,dx,  \label{eqInt.5}
\end{equation}
with $F\left( v_+\right) =\int_{0}^{v_+}f\left( s\right) \,ds$. Then the
well-known mountain-pass lemma \cite{AmRa73} implies that%
\begin{equation}
c_{\varepsilon}=\inf_{h\in\Gamma}\max_{t\in\left[ 0,1\right]
}J_{\varepsilon}\left( h\left( t\right) \right)  \label{eqInt.6}
\end{equation}
is a positive critical value of $J_{\varepsilon}$, where $\Gamma$ is the set
of all continuous paths joining the origin and a fixed nonzero element $e\in
W^{1,m}\left( \Omega\right) $ such that $e\geq0$ and $J_{\varepsilon}\left(
e\right) \leq0$. It turns out $c_{\varepsilon}$ can also be characterized as
follows:%
\begin{equation*}
c_{\varepsilon}=\inf_{u\in M_{\varepsilon}}J_{\varepsilon}\left( u\right)
\end{equation*}
with 
\begin{equation*}
M_{\varepsilon}=\left\{ \,u\in W^{1,m}\left( \Omega\right) \mathrel{;}u\geq
0,\;u\not \equiv 0,\;\int_{\Omega}\left( \varepsilon^{m}\left\vert \nabla
u\right\vert ^{m}+u^{m}\right) \,dx=\int_{\Omega}f\left( u\right)
u\,dx\,\right\}
\end{equation*}
or%
\begin{equation}
c_{\varepsilon}=\inf\left\{ \,M\left[ u\right] \mid u\in W^{1,m}\left(
\Omega\right) ,\;u\not \equiv 0\text{ and }u\geq0\text{ in }\Omega\,\right\}
\label{eqInt.7}
\end{equation}
with%
\begin{equation*}
M\left[ u\right] =\sup_{t\geq0}J_{\varepsilon}\left( tu\right) .
\end{equation*}
Hence $c_{\varepsilon}$ is the least positive critical value and a critical
point $u_{\varepsilon}$ of $J_{\varepsilon}$ with critical value $%
c_{\varepsilon}$ is called a least-energy solution. Notice also that if we
let%
\begin{equation*}
c_{\ast}=\text{I}(w)=\frac{1}{m}\int_{\mathbb{R}^{N}}\left( \left\vert
\nabla w\right\vert ^{m}+w^{m}\right) \,dx-\int_{\mathbb{R}^{N}}F\left(
w\right) \,dx,
\end{equation*}
where $w$ is the unique least energy solution of (\ref{IntroEq5}), then $%
c_{\ast}$ can also be characterized as 
\begin{equation}
c_{\ast}=\inf\left\{ \,M_{\ast}\left[ v\right] \mid v\in W^{1,m}\left( 
\mathbb{R}^{N}\right) ,\;v\not \equiv 0\text{ and }v\geq0\text{ in }\mathbb{R%
}^{N}\,\right\}  \label{eqInt.8}
\end{equation}
with%
\begin{equation*}
M_{\ast}\left[ v\right] =\sup_{t\geq0}I\left( tv\right) .
\end{equation*}
We refer to Lemma 2.1 of \cite{LiZh04c} for the above characterizations.

\ 

Next we consider the following problem:

$v\in W^{1,m}\left( \mathbb{R}^{N}_{+}\right) $ with $\mathbb{R}%
^{N}_{+}=\left\{ \left( x_{1},\cdots, x_{N}\right) \in\mathbb{R}^{N},
x_{N}\geq0\right\} $ and satisfies 
\begin{equation}
\left\{ \begin{aligned} &\Delta_{m}v-v^{m-1}+f\left(v\right) =0, & &
v>0\text{ in }\mathbb{R}^N_+, \\ &\frac{\partial v}{\partial x_N}=0 & &
\text{on }x_N=0. \end{aligned}\right.  \label{eqInt.9}
\end{equation}
The solutions of (\ref{eqInt.9}) can be characterized as critical points of
the functional defined over $W^{1,m}\left( \mathbb{R}_{+}^{N}\right) $ as
follows. 
\begin{equation*}
I_{\mathbb{R}^{N}_{+}}(\tilde v)=\frac{1}{m}\int_{\mathbb{R}^{N}_{+}}\left|
\nabla\tilde v|^{m}+{\tilde v}^{m}\right) \,dx-\int_{\mathbb{R}%
^{N}_{+}}F(\tilde v_+) \,dx.
\end{equation*}
Similarly as above the least positive critical value $C_{*}$ corresponding
to least energy solutions of (\ref{eqInt.9}) can be characterized as 
\begin{equation}
C_{*}=\inf_{\tilde v\in W^{1,m}\left( \mathbb{R}^{N}_{+}\right) ,\tilde
v\geq0, \tilde v\not \equiv 0}\sup_{t>0}I_{\mathbb{R}^{N}_{+}}(t\tilde v)
\label{eqInt.10}
\end{equation}
and moreover 
\begin{equation}
C_{*}=\frac12 c_{*}  \label{eqInt.11}
\end{equation}
due to the boundary condition in (\ref{eqInt.9}) and the fact that $w$ is
radial and hence $\frac{\partial w}{\partial x_{N}}=0.$ We also refer to
Lemma 2.1 of \cite{LiZh04c} for the above characterization of $C_{*}$. In
Theorem 1.3 of \cite{LiZh04c}, we proved the following theorem.

\begin{theorem}
\label{ThmA}Under hypotheses \textup{(\ref{hypothesis1})--(\ref{hypothesis5})%
}, let $u_{\varepsilon}$ be a least-energy solution of \textup{(\ref{eqInt.1}%
)}. Then all local maximum points(if more than one) of $u_{\varepsilon}$
aggregate to a global maximum point $P_{\varepsilon}$ at a rate of $%
o(\varepsilon)$ and $dist(P_{\varepsilon},\partial\Omega)/\varepsilon$$%
\rightarrow 0$ as $\varepsilon\rightarrow 0^+$, where $dist(\cdot,\cdot)$ is
the general distance function. Moreover, we have the following upper-bound
estimate for $c_{\varepsilon}$ as $\varepsilon\rightarrow0^{+}$\textup{:}%
\begin{equation}
c_{\varepsilon}\leq\varepsilon^{N}\left\{ \frac{1}{2}c_{\ast}-\left(
N-1\right) \max_{P\in\partial\Omega}H\left( P\right) \gamma\varepsilon
+o\left( \varepsilon\right) \right\} ,  \label{eqInt.12}
\end{equation}
where $H\left( P\right) $ denotes the mean curvature of $\partial\Omega$ at $%
P$, $\gamma>0$ is a positive constant given by 
\begin{equation}
\gamma=\frac{1}{N+1}\int_{\mathbb{R}_{+}^{N}}\left\vert w^{\prime}\left(
\left\vert z\right\vert \right) \right\vert ^{m}z_{N}\,dz.  \label{eqInt.13}
\end{equation}
\end{theorem}

Our goal in this paper is to locate the position on $\partial\Omega$ where
the global maximum point $P_{\varepsilon}$ of $u_{\varepsilon}$ in $%
\overline{\Omega}$ approaches, provided $\varepsilon$ is sufficiently small.
For the case $m=2$, Ni and Takagi \cite{NiTa93} located the peak by
linearizing the equation $d\Delta u-u+f\left( u\right) =0$ around the ground
state $w$. But this method fails for our problem with $m\neq2$ due to the
strong nonlinearity of the $m$-Laplacian operator $\Delta_{m}u=\limfunc{div}%
(\left\vert \nabla u\right\vert ^{m-2}\nabla u)$. So we have to use the
intrinsic variational method created by Del Pino and Felmer in \cite{DeFe99}
to attack it. We also give a complete proof of the exponential decay of the
least-energy solution $u_{\varepsilon}.$ We remark that our proof is
complete and does not require the non-degeneracy of the unique radial least
energy solution $w$ as stated in Proposition \ref{IntroProp1}, and hence it
is different from Ni's and Takagi's work \cite{NiTa91}. Now our results can
be stated as follows:

\begin{theorem}
\label{ThmB}Under hypotheses \textup{(\ref{hypothesis1})--(\ref{hypothesis5})%
}, let $u_{\varepsilon}$ be a least-energy solution of \textup{(\ref{eqInt.1}%
)} and $\tilde P_{\varepsilon}\in\partial\Omega$ with $dist(P_{\varepsilon},%
\tilde P_{\varepsilon})=dist(P_{\varepsilon},\partial\Omega)$. Then as $%
\varepsilon\rightarrow 0^+$, after passing to a sequence $\tilde
P_{\varepsilon}$ approaches $\bar{P}\in\partial\Omega$ with

\begin{enumerate}
\item \renewcommand{\theenumi}{\roman{enumi}}

\item \label{ThmB(1)}$\displaystyle H\left( \bar{P}\right) =\max
_{P\in\partial\Omega}H\left( P\right) $, where $H\left( P\right) $ denotes
the mean curvature of $\partial\Omega$ at $P$ as stated before, and moreover

\item \label{ThmB(2)}the associated critical value $c_{\varepsilon}$ can be
estimated as $\varepsilon\rightarrow0^{+}$ as follows\/\textup{:} 
\begin{equation}
c_{\varepsilon}=\varepsilon^{N}\left\{ \frac{1}{2}c_{\ast}-\left( N-1\right)
H\left( \bar{P}\right) \gamma\varepsilon+o\left( \varepsilon \right)
\right\} ,  \label{eqInt.14}
\end{equation}
where $c_{\ast}$, $\gamma$ are as stated in Theorem \textup{\ref{ThmA}}.
\end{enumerate}
\end{theorem}

The organization of this paper is as follows: In Section \ref{Som}, we will
prove some lemmas which will be used in proving Theorem \ref{ThmB}. The
proof of Theorem \ref{ThmB} will be given in Section \ref{ProB}.

\section{\label{Som}Some lemmas and exponential decay of $u_{\protect%
\varepsilon}$}

First we prove the following lemma related to exponential decay of the
least-energy solution $u_{\varepsilon}$.

\begin{lemma}
\label{LemSom.2}Let $\varepsilon$ be sufficiently small and that the
least-energy solution $u_{\varepsilon}$ achieves its global maximum at some
point $P_{\varepsilon}$. Then there exist two positive constants $c_{3}$ and 
$c_{4}$ independent of $u_{\varepsilon}$ or $\varepsilon$ such that%
\begin{equation}
\begin{aligned} u_{\varepsilon}\left( x\right) &\leq c_{3}\exp\left\{
-c_{4}\left\vert x-P_{\varepsilon}\right\vert /\varepsilon\right\}\\
\left\vert\nabla u_{\varepsilon}(x)\right\vert &\leq c_3
\varepsilon^{-1}\exp\{-c_4|x-P_{\varepsilon}|/ \varepsilon\}. \end{aligned}
\label{eqSom.3}
\end{equation}
\end{lemma}

Before beginning to prove this lemma, we give a remark on it.

\begin{remark}
\label{RemSom.1}For the case $m=2$, under the assumption of non-degeneracy
of the linearized operator $\Delta-1+f^{\prime}\left( w\right) $, where $w$
is the unique ground state of \textup{(\ref{IntroEq5})}, Ni and Takagi \cite%
{NiTa93} showed that $u_{\varepsilon}\left( x\right) $ can be written as%
\begin{equation}
u_{\varepsilon}\left( x\right) =w\left( x\right) +\varepsilon\phi _{1}\left(
x\right) +o\left( \varepsilon\right)  \label{eqSom.4}
\end{equation}
and $\phi_{1}\left( x\right) $ enjoys the exponential-decay property (\cite%
{NiTa93}). Clearly we cannot derive exponential decay of $u_{\varepsilon
}\left( x\right) $ as stated in Lemma (\ref{LemSom.2}) from (\ref{eqSom.4})
even though both $w\left( x\right)$ and $\varepsilon\phi _{1}\left( x\right)$
have exponential decay property.
\end{remark}

\begin{proof}[Proof of Lemma \textup{\protect\ref{LemSom.2}}]
Since $\partial\Omega$ is a smooth compact submanifold of $R^{N},$ it
follows from the tubular neighborhood theorem \cite{Hirsch} that there
exists a constant $\omega\left( \Omega\right) >0$ which depends only on $%
\Omega$ such that $\Omega _{I}=\left\{ x\in\overline{\Omega},d\left(
x,\partial\Omega\right) <\omega\left( \Omega\right) \right\} $ is
diffeomorphic to the inner normal bundle 
\begin{equation*}
\left( \partial\Omega\right) _{I}^{N}=\left\{ \left( x,y\right)
:x\in\partial\Omega,y\in\left( -\omega\left( \Omega\right) ,0\right]
\nu_{x}\right\} ,
\end{equation*}
here $\nu_{x}$ is the unit outer normal of $\partial\Omega$ at $x,$ and the
diffeomorphism is defined as follows: $\forall x\in\Omega_{I},$ there exists
an unique $\hat{x}\in\partial\Omega$ such that $d\left( x,\hat{x}\right)
=d\left( x,\partial\Omega\right) ,$ then $\Phi_{\ast}:x\longrightarrow
\left( \hat{x},-d\left( x,\hat{x}\right) \nu_{\hat{x}}\right) .$ Moreover
this diffeomorphism satisfies $\Phi_{\ast}|_{\partial\Omega}=\text{Identity}%
. $ Similarly, let $\Omega_{O}=\left\{ x\in\mathbb{R}^{N}\setminus\Omega
,d(x,\partial\Omega)<\omega\left( \Omega\right) \right\} .$ Then $\Omega_{O}$
is diffeomorphic to the outer normal bundle 
\begin{equation*}
\left( \partial\Omega\right) _{O}^{N}=\left\{ \left( x,y\right)
:x\in\partial\Omega,y\in\left[ 0,\omega\left( \Omega\right) \right)
\nu_{x}\right\} ,
\end{equation*}
and the diffeomorphism is given as follows. $\forall x\in\Omega_{O},$ there
exists an unique $\bar{x}\in\partial\Omega$ such that $d(x,\bar{x})=d\left(
x,\partial\Omega\right) ,$ and then $\Phi_{\#}:x\longrightarrow\left( \bar{x}%
,d\left( x,\bar{x}\right) \nu_{\bar{x}}\right) $ and $\Phi
_{\#}|_{\partial\Omega}=\text{Identity}.$ Note that $\left( \partial
\Omega\right) _{I}^{N}$ is clearly diffeomorphic to $\left( \partial
\Omega\right) _{O}^{N}$ via the following reflection $\Phi^{\ast}:\left(
\partial\Omega\right) _{I}^{N}\longrightarrow\left( \partial\Omega\right)
_{O}^{N}$ defined by $\Phi^{\ast}\left( \left( x,y\right) \right) =\left(
x,-y\right) .$ Therefore, $\Phi=\Phi_{\ast}^{-1}\circ\Phi^{\ast-1}\circ
\Phi_{\#}:\Omega_{O}\longrightarrow\Omega_{I}$ is the desired diffeomorphism
and $\Phi|_{\partial\Omega}=\text{Identity}.$ Moreover, if we let $%
x=\Phi(z)=\left( \Phi_{1}(z),\cdots,\Phi_{N}(z)\right) ,$ $z\in\Omega_{O},$
and $z=\Psi(x)=\Phi^{-1}(x)=\left( \Psi_{1}(x),\cdots,\Psi_{N}(x)\right) ,$ $%
x\in\Omega_{I},$ $g_{ij}=\sum_{k=1}^{N}\frac{\partial\Phi_{k}}{\partial z_{i}%
}\frac{\partial\Phi_{k}}{\partial z_{j}},$ $g^{ij}=\sum_{k=1}^{N}\frac{%
\partial\Psi_{i}}{\partial x_{k}}\frac{\partial\Psi_{j}}{\partial x_{k}}%
\left( \Phi\left( z\right) \right) ,$ we have $g_{ij}|_{\partial\Omega
}=g^{ij}|_{\partial\Omega}=\delta_{ij}$ with $\delta_{ij}$ being the
Kronecker symbol. Denote $G=\left( g^{ij}\right) $ and $A=G-I$ with $I$
being the $N\times N$ identity matrix, $g(x)=\text{det}\left( g_{ij}\right) $
and $\hat{u}_{\varepsilon}(x)=u_{\varepsilon}\left( \Phi\left( x\right)
\right) $ for $x\in\Omega_{O}.$ Then $\hat{u}_{\varepsilon}(x)$ satisfies
the following equations: 
\begin{equation}
\left\{ \begin{aligned} &\varepsilon^m \text{L}\hat
u_{\varepsilon}-\sqrt{g}\hat u_{\varepsilon}^{m-1}+ \sqrt{g}f\left(\hat
u_{\varepsilon}\right)&=0,\quad\hat u_{\varepsilon}>0\quad \text{in}\quad
\Omega_O\\ \nonumber&\frac{\partial \hat u_{\varepsilon}}{\partial\nu}=0,
\quad\text{on}\quad\partial\Omega, \end{aligned}\right.  \label{equation*1}
\end{equation}
where 
\begin{equation*}
\begin{aligned} \text{L}\hat u
_{\varepsilon}&=\sum_{i=1}^N\frac{\partial}{\partial
x_i}\left\{\left[\sum^N_{s,l=1}g^{sl} \frac{\partial\hat
u_{\varepsilon}}{\partial x_s}\frac{\partial\hat u_{\varepsilon}} {\partial
x_l}\right]^{\frac{m-2}{2}} \sqrt{g}\sum_{j=1}^N g^{ij}\frac{\partial\hat
u_{\varepsilon}}{\partial x_j}\right\}\\
&=\text{Tr}\left\{D\left(\left[\nabla\hat u_{\varepsilon} G\left(\nabla\hat
u_{\varepsilon}\right)^T\right]^{\frac{m-2}{2}} \sqrt{g}\left(\nabla\hat
u_{\varepsilon}\right)G\right)\right\}, \end{aligned}
\end{equation*}
where Tr means taking the trace of a square matrix.

For $0<\tilde{\gamma}\leq\omega(\Omega)$, let $\Omega_{O\tilde{\gamma}%
}=\left\{ x\in\overline{\Omega_{O}},d(x,\partial\Omega)<\tilde{\gamma }%
\right\} .$ We know $\Vert A\Vert_{C^{0}}$ can be made arbitrarily small by
making $\tilde{\gamma}$ sufficiently small. Next we define 
\begin{equation*}
\bar{u}_{\varepsilon}=\left\{ \begin{aligned} &u_{\varepsilon}(x),\quad
x\in\Omega\\ &\hat u_{\varepsilon}(x),\quad x\in\Omega_O, \end{aligned}%
\right.
\end{equation*}%
\begin{equation*}
\tilde{g}_{ij}=\left\{ \begin{aligned} &\delta_{ij},\quad x\in
\overline{\Omega}\\ &g_{ij},\quad x\in \Omega_O,\\ \end{aligned}\right.
\end{equation*}%
\begin{equation*}
\tilde{g}^{ij}=\left\{ \begin{aligned} &\delta_{ij},\quad x\in
\overline{\Omega}\\ &g^{ij},\quad x\in \Omega_O,\\ \end{aligned}\right.
\end{equation*}
and $\tilde{A}(x,\xi)=\left( \tilde{A}_{1}(x,\xi),\cdots,\tilde{A}%
_{N}(x,\xi)\right) $ for $\xi=\left( \xi_{1},\cdots,\xi_{N}\right) $ with 
\begin{equation*}
\tilde{A}_{i}(x,\xi)=\left[ \sum_{s,l=1}^{N}\tilde{g}^{sl}\xi_{s}\xi _{l}%
\right] ^{\frac{m-2}{2}}\sqrt{\tilde{g}}\sum_{j=1}^{N}\tilde{g}^{ij}\xi_{j}
\end{equation*}
and $\tilde{g}=\text{det}(\tilde{g}_{ij}),$ $B(x,u)=\sqrt{\tilde{g}}\left(
-u^{m-1}+f(u)\right) .$ Then $\bar{u}_{\varepsilon}(x)$ satisfies 
\begin{equation}
\varepsilon^{m}\limfunc{div}\left( \tilde{A}(x,\nabla\bar {u}%
_{\varepsilon})\right) +B(x,\bar{u}_{\varepsilon})=0\quad\text{in}\quad%
\overline{\Omega}\bigcup\Omega_{O}  \label{equation2}
\end{equation}
in the weak sense.

For any ball $B_{r}(x_{0})\subset\overline{\Omega}\bigcup\Omega_{O}$ with
radius $r$ and center $x_{0}\in\Omega$, let $\rho=|x-x_{0}|.$ Then for any
smooth increasing function $\phi=\phi(\rho)$ we have 
\begin{equation*}
\begin{aligned} &\left[(\nabla\phi)
G(\nabla\phi)^T\right]^{\frac{m-2}2}\sqrt{g}(\nabla\phi)G\\
&=\left\vert\nabla\phi(I+A)(\nabla\phi)^T\right\vert^{\frac{m-2}2}\sqrt{%
\text{det}(I+A)^{-1}}\nabla\phi(I+A)\\
&=\left\vert\nabla\phi\right\vert^{m-2}\nabla\phi+\int^1_0\frac{d}{dt}\left(
\left\vert\nabla\phi(I+tA)(\nabla\phi)^T\right\vert^{\frac{m-2}2}\sqrt{%
\text{det}(I+tA)^{-1}}\nabla\phi(I+tA)\right)dt\\
&=\left\vert\nabla\phi\right\vert^{m-2}\nabla\phi\\
&+\frac{m-2}2\int^1_0\left\vert\nabla\phi(I+tA)(\nabla\phi)^T\right\vert^{%
\frac{m-4}2}
\left((\nabla\phi)A(\nabla\phi)^T\right)\sqrt{\text{det}(I+tA)^{-1}}\nabla%
\phi(I+tA)dt\\ &+\int^1_0\left\vert\nabla\phi(I+t
A)(\nabla\phi)^T\right\vert^{\frac{m-2}2}\frac{\frac{d}{dt}
\left(\text{det}(I+tA)^{-1}\right)}
{2\sqrt{\text{det}(I+tA)^{-1}}}\nabla\phi(I+tA)dt\\
&+\int^1_0\left\vert\nabla\phi(I+tA)(\nabla\phi)^T\right\vert^{\frac{m-2}2}%
\sqrt{\text{det}(I+tA)^{-1}}(\nabla\phi)Adt. \end{aligned}
\end{equation*}
Therefore 
\begin{equation}
\begin{aligned}
\text{Tr}&\left[D\left(\left[(\nabla\phi)G(\nabla\phi)^T\right]^{%
\frac{m-2}2} \sqrt{g}(\nabla\phi)G\right)\right]\\ &\leq \frac{3}2
\left\vert\left(\left\vert\phi^{\prime}\right\vert^{m-2}\phi^{\prime}%
\right)^{\prime}\right\vert +
\frac{3(N-1)}{2\rho}\left\vert\phi^{\prime}\right\vert^{m-2}\phi^{\prime}+
K\left\vert\phi^{\prime}\right\vert^{m-2}\phi^{\prime} \end{aligned}
\label{equation3}
\end{equation}
by taking $\tilde\gamma$ sufficiently small, here $K>0$ is a constant
depending only on $\Psi,$ hence only on $\Omega$ and $\phi^{\prime} =\frac{%
d\phi(\rho)}{d\rho}.$

{}From now on $\tilde\gamma=\tilde\gamma(\Omega)$ is fixed such that (i)\quad%
$\frac{3}4\leq\sqrt{g}\leq\frac{5}4$, (ii) (\ref{equation3}) holds for any
smooth increasing radial function $\phi(\rho)$ and (iii) $\frac{3}%
4|\xi|^{m}\leq\tilde A(x,\xi)\cdot\xi\leq\frac{5}4|\xi|^{m}$ for any $\xi
=(\xi_{1},\cdots,\xi_{N}).$ Denote $\Omega^{\tilde\gamma}=\Omega\cup%
\Omega_{O\tilde\gamma}.$

Let $\Omega_{\varepsilon}=\frac{1}{\varepsilon}\left(\Omega-P_{\varepsilon}%
\right)$ and $u^{\varepsilon}(x)
=u_{\varepsilon}(P_{\varepsilon}+\varepsilon x)$ for $x\in\Omega_{%
\varepsilon}.$ Then $u^{\varepsilon}$ is a solution to the following
problem: 
\begin{equation}
\left\{ \begin{aligned}&\Delta_mu^{\varepsilon}-\left(u^{\varepsilon}%
\right)^{m-1}+f(u^{\varepsilon})=0, u^{\varepsilon}>0 \quad\text{in }\quad
\Omega_{\varepsilon} \\ &\frac{\partial u^{\varepsilon}}{\partial
n}=0,\quad\text{on}\quad\partial\Omega_{\varepsilon}, \end{aligned}\right.
\label{aaequation4}
\end{equation}
where $n$ is the unit outer normal of $\partial\Omega_{\varepsilon}.$
Similarly, let $\Omega^{\tilde\gamma}_{\varepsilon}=\frac{1}{\varepsilon}%
\left(\Omega^{\tilde\gamma}-P_{\varepsilon}\right)$ and $\bar
u^{\varepsilon}(x) =\bar u_{\varepsilon}(P_{\varepsilon}+\varepsilon x)$ for 
$x\in\Omega^{\tilde\gamma}_{\varepsilon}.$ Since $\bar u^{\varepsilon}$
converges to the unique radial least-energy solution $w$ of (\ref{IntroEq5})
in $C^{1}_{loc}(\mathbb{R}^{N}) \cap W^{1,m}(\mathbb{R}^{N})$ as $%
\varepsilon\rightarrow0^{+}$ (see the proof of Theorem 1.2 of \cite{LiZh04c}%
) and $w$ satisfies: 
\begin{equation*}
\begin{aligned} &\text{(i)}\quad w\quad \text{is radial,
i.e.,}w(x)=w(|x|)=w(r)>0\\ &\text{(ii)} \lim_{r\rightarrow\infty}w(r)r^{
\frac{N-1}{m(m-1)}}e^{\left(\frac{1}{m-1}\right)^{\frac{1}{m}}r}=C_0>0
\end{aligned}
\end{equation*}
(see Theorem 1 of \cite{LiZh04a}) which yields $w(r)\leq\kappa e^{-\mu r}$
for a constant $\kappa>0$ and $\mu=\left( \frac{1}{m-1}\right) ^{\frac{1}{m}%
}.$ First we fix a constant $\eta>0$ such that $\frac{1}{8} t^{m-1}>f(t)$
for $t\in(0,\eta].$ From hypothesis (\ref{hypothesis4}) it follows that such
an $\eta$ exists. Then there exist $\varepsilon_{0}>0$ sufficiently small
and $R_{0}$ sufficiently large such that $4\kappa\exp\{-\mu R_{0}\} <\eta$
and $\|\bar
u^{\varepsilon}-w\|_{C^{0}(B_{R_{0}}(0)\cap\Omega_{\varepsilon})}\leq
\kappa\exp\{-\mu R_{0}\},$ which yields 
\begin{equation*}
u^{\varepsilon}\vert_{(\partial
B_{R_{0}}(0))\cap\Omega_{\varepsilon}}\leq2\kappa\exp\{-\mu R_{0}\}.
\end{equation*}
Note that 
\begin{equation*}
\left\{%
\begin{array}{l}
\Delta_m u^{\varepsilon}-\frac78(u^{\varepsilon})^{m-1}=\frac{1}{8}%
(u^{\varepsilon})^{m-1}-f(u^{\varepsilon})>0 \quad\text{in}%
\quad\Omega_{\varepsilon}\setminus B_{R_0}(0), \\[2pt] 
\frac{\partial u^{\varepsilon}}{\partial n}=0\quad\text{on}%
\quad\partial\Omega_{\varepsilon}\setminus B_{R_0}(0), \\[2pt] 
u^{\varepsilon}\leq2\kappa\exp\{-\mu R_{0}\}\quad\text{on}\quad\partial
B_{R_0}(0)\cap\overline{\Omega_{\varepsilon}}.%
\end{array}%
\right.
\end{equation*}
Then we have 
\begin{equation*}
u^{\varepsilon}(x)\leq2\kappa\exp\{-\mu R_{0}\},\quad\text{for} \quad
x\in\Omega_{\varepsilon}\setminus B_{R_{0}}(0)
\end{equation*}
due to the strong maximum principle (\cite{Vazquez}). We get by scaling back
that 
\begin{equation*}
u_{\varepsilon}\vert_{\Omega\setminus B_{\varepsilon R_{0}}(0)}\leq2\kappa
\exp\{-\mu R_{0}\}
\end{equation*}
and 
\begin{equation}
\begin{aligned} u_{\varepsilon}(x)&\leq
w\left(\frac{|x|}{\varepsilon}\right) +\kappa \exp\{-\mu R_0\} \leq \kappa
\exp\{- \frac{ \mu |x|}{\varepsilon}\} +\kappa \exp\{-\mu R_0\}\\ &\leq
2\kappa \exp\{-\frac{\mu |x|}{\varepsilon}\} \end{aligned}  \label{equation6}
\end{equation}
for $x\in\Omega\cap B_{\varepsilon R_{0}}(0).$

{}From definition of $\bar{u}_{\varepsilon}$ we know 
\begin{equation*}
\bar{u}_{\varepsilon}(x)\leq2\kappa\exp\{-\frac{\mu\left(|x|-2\text{dist}
\left(P_{\varepsilon},\partial\Omega\right)\right)}{\varepsilon}\}\leq
4\kappa\exp\{-\frac{\mu|x|}{\varepsilon}\} \quad\text{for}\quad
x\in\Omega^{\tilde\gamma}\cap B_{\varepsilon R_{0}}(0)
\end{equation*}
for $\varepsilon\in(0,\varepsilon_0]$ with $\varepsilon_0$ sufficiently
small due to the fact $\text{dist}(P_{\varepsilon},\partial\Omega)=o(%
\varepsilon)$ as $\varepsilon\rightarrow 0^+.$ Note that 
\begin{equation*}
\sup_{\Omega^{\tilde\gamma} \setminus B_{\varepsilon R_{0}}(0)}\bar{u}%
_{\varepsilon}\leq4\kappa\exp\{-\mu R_{0}\}.
\end{equation*}
Choice of $R_{0}$ and $\tilde{\gamma}$ tells us for any $0<t\leq4\kappa
\exp\{-\mu R_{0}\}$ 
\begin{equation*}
B(x,t)=\sqrt{\tilde{g}}\left( -t^{m-1}+f(t)\right) \leq-\frac{1}{2}t^{m-1}.
\end{equation*}
$\forall x_{0}\in\Omega\setminus B_{\varepsilon R_{0}}(0)$ and $%
B_{r}(x_{0})\subset\Omega^{\tilde\gamma}\setminus B_{\varepsilon R_{0}}(0),$
define 
\begin{equation}
\begin{aligned} \phi(x)&=\phi(\rho)=\phi(|x-x_0|)\\\nonumber &=
\left\{\sup_{\Omega^{\tilde\gamma}\setminus B_{\varepsilon R_0}(0)} \bar
u_{\varepsilon}\right\}
\frac{\text{cosh}\left(\frac{\lambda_*\rho}{\varepsilon}\right)}
{\text{cosh}\left(\frac{\lambda_*r}{\varepsilon}\right)}, \end{aligned}
\label{equation5}
\end{equation}
where $\lambda_{\ast}>0$ is a constant to be determined later. Simple
calculations show that 
\begin{equation*}
\begin{aligned} &(i)\quad
\phi^{\prime}(\rho)>0\quad\text{and}\quad\left(\left\vert\phi^{\prime}\right%
\vert^{m-2} \phi^{\prime}\right)^{\prime}>0;\\ &(ii)\quad\begin{array}{l}
\varepsilon^m\left[\frac{3}{2}\left(\left\vert\phi^{\prime}\right\vert^{m-2}%
\phi^{\prime}\right)^{\prime} +\frac{3(N-1)}{2\rho}
\left\vert\phi^{\prime}\right\vert^{m-2}\phi^{\prime}+ K
\left\vert\phi^{\prime}\right\vert^{m-2}\phi^{\prime}\right]-\frac{1}{2}
\phi^{m-1}\\ =\left[\frac32\left(m-1\right)\left(\lambda_{\ast}\right)^m
\left(\tanh\left(\frac{\lambda_{\ast}\rho}{\varepsilon}\right)\right)^{m-2}
+\frac32\left(\lambda_{\ast}\right)^{m}
\left(\frac{\tanh\left(\frac{\lambda_{\ast}\rho}{\varepsilon}\right)}
{\left(\frac{\lambda_{\ast}\rho}{\varepsilon}\right)^{\frac1{m-1}}}%
\right)^{m-1}\right.\\ \left.+\varepsilon\left(\lambda_{\ast}\right)^{m-1}K
\left(\tanh\left(\frac{\lambda_{\ast}\rho}{\varepsilon}\right)\right)^{m-1}
-\frac12\right]\phi^{m-1}\\ \leq 0 \end{array} \end{aligned}
\end{equation*}
for any $0<\lambda_{\ast}\leq\hat{\lambda},$ where $\hat{\lambda}>0$ is a
small constant depending only on $m$ and $\Omega$ through $K$. We remark
that we have used the fact $\max_{r\in\lbrack0,\infty)}\frac{\text{tanh}r}{%
r^{\frac{1}{m-1}}}<\infty$ for $m\geq2.$ From now on we choose $\lambda
_{\ast}=\hat{\lambda}.$

Therefore we have 
\begin{equation*}
\begin{aligned} &\varepsilon^m\operatorname*{div}(\tilde A(x,\nabla\bar
u_{\varepsilon})) -\frac{1}{2}\bar u_{\varepsilon}^{m-1}\geq 0\quad{in}\quad
B_r(x_0),\\ &\varepsilon^m\operatorname*{div}(\tilde
A(x,\nabla\phi))-\frac{1}{2}\phi^{m-1}\leq 0\quad{in}\quad B_r(x_0).
\end{aligned}
\end{equation*}
Clearly 
\begin{equation*}
\phi|_{\partial B_{r}(x_{0})}\geq\bar{u}_{\varepsilon}|_{\partial
B_{r}(x_{0})}.
\end{equation*}
Then from the Comparison Theorem (Theorem 10.1 of \cite{Serrin}) it follows
that 
\begin{equation*}
\phi(x)\geq\bar{u}_{\varepsilon}(x)\quad{in}\quad B_{r}(x_{0}).
\end{equation*}
In particular, $\phi(x_{0})\geq\bar{u}_{\varepsilon}(x_{0}).$ Thus we get 
\begin{equation*}
u_{\varepsilon}(x_{0})\leq\left( \sup_{\Omega^{\tilde\gamma} \setminus
B_{\varepsilon R_{0}}(0)}\bar {u}_{\varepsilon}\right) \exp\{-\frac{%
\lambda_{\ast}r}{\varepsilon}\}.
\end{equation*}
Choosing $r=d\left(x_{0},\partial\left(\Omega^{\tilde\gamma} \setminus
B_{\varepsilon R_{0}}(0)\right)\right)$ we get 
\begin{equation*}
u_{\varepsilon}(x_{0})\leq4\kappa\exp\{-\mu R_{0}-\frac{\lambda_{\ast}r}{%
\varepsilon}\}\leq2\kappa\exp\{-\frac{\tilde{\lambda}(\varepsilon R_{0}+r)}{%
\varepsilon}\}
\end{equation*}
with $\tilde{\lambda}=\min\{\mu,\lambda_{\ast}\}.$ Note that $x_{0}$ belongs
to one of the following two cases: 
\begin{equation*}
\begin{aligned} &(i)\quad
d\left(x_0,\partial\left(\Omega^{\tilde\gamma}\setminus B_{\varepsilon
R_0}(0)\right)\right)=d\left(x_0,\partial B_{\varepsilon R_0}(0)\right),\\
&(ii)\quad d\left(x_0,\partial\left(\Omega^{\tilde\gamma}\setminus
B_{\varepsilon R_0}(0)\right)\right)=d\left(x_0,
\partial\Omega^{\tilde\gamma}\right). \end{aligned}
\end{equation*}
For case (i) we have $d(x_{0},P_{\varepsilon})\leq\varepsilon R_{0}+r$ and
therefore 
\begin{equation}
u_{\varepsilon}(x_{0})\leq4\kappa\exp\{-\frac{\tilde{\lambda}%
d(x_{0},P_{\varepsilon})}{\varepsilon}\}.  \label{equation7}
\end{equation}
For case (ii) we have $r\geq\tilde{\gamma}$ and thus 
\begin{equation}
\begin{aligned} u_{\varepsilon}(x_0)&\leq 4\kappa\exp\{-\tilde\lambda
\frac{\varepsilon R_0+r}{\varepsilon}\}\leq
4\kappa\exp\{-\frac{\tilde\lambda\tilde\gamma}{\varepsilon}\}\\ &\leq
4\kappa\exp\{-\tilde\lambda\frac{\tilde\gamma}{\text{diam}(\Omega)}\cdot%
\frac{d(x_0,P_{\varepsilon})} {\varepsilon}\}, \end{aligned}
\label{equation8}
\end{equation}
Combining (\ref{equation6}), (\ref{equation7}) and (\ref{equation8})
together and letting $\tilde{c_{3}}=4\kappa,$ $\tilde{c_{4}}=\min\{\mu,%
\tilde{\lambda},\frac {\tilde{\lambda}\tilde{\gamma}}{\text{diam}(\Omega)}\}$
yields 
\begin{equation}
u_{\varepsilon}(x)\leq \tilde{c_{3}}\exp\{-\frac{\tilde{c_{4}}%
|x-P_{\varepsilon}|}{\varepsilon}\}.  \label{adduexp}
\end{equation}

Next we show the estimate for $\left\vert\nabla u_{\varepsilon}\right\vert$
holds. First from (\ref{aaequation4}) it follows that 
\begin{equation}
\Delta_mu^{\varepsilon}=\left(u^{\varepsilon}\right)^{m-1}-f(u^{%
\varepsilon}), u^{\varepsilon}>0 \quad\text{in }\quad \Omega_{\varepsilon}
\label{addequation4}
\end{equation}
For $x\in\Omega_{\varepsilon}$ and dist$(x,\partial\Omega_{\varepsilon})\geq
1$, consider (\ref{addequation4}) in the unit ball centered at $x$, i.e., $%
B_1(x)$. Then by an $C^{1,\alpha}$ estimate (see \cite{Tol}, for example)
there exists two constants $C>0$ and $\alpha^*\in (0,1)$ which are
independent of $\varepsilon$ such that 
\begin{equation}
\begin{aligned} \|u^{\varepsilon}\|_{C^{1,\alpha^*}(B_{\frac 1 2}(x))}&\leq
C\left(\|u^{\varepsilon}\|_{L^{\infty}(B_1(x))}
+\|\left(u^{\varepsilon}\right)^{m-1}-f(u^{\varepsilon})\|^{%
\frac{1}{m-1}}_{L^{\infty}(B_1(x))}\right)\\ &\leq
c_3^*\exp\{-c_4^*|x-P_{\varepsilon}|\}, \end{aligned}  \label{addexp1}
\end{equation}
where we have used (\ref{adduexp}) and the fact that $u^{\varepsilon}(x)
=u_{\varepsilon}(P_{\varepsilon}+\varepsilon x)$ for $x\in\Omega_{%
\varepsilon}.$ Especially we have 
\begin{equation}
|\nabla u^{\varepsilon}(x)|\leq c_3^*\exp\{-c_4^*|x-P_{\varepsilon}|\},
\label{addexp2}
\end{equation}
for $x\in\Omega_{\varepsilon}$ and dist$(x,\partial\Omega_{\varepsilon})\geq
1$. For $x\in\Omega_{\varepsilon}$ with dist$(x,\partial\Omega_{%
\varepsilon})<1$. Let $x_0\in\partial\Omega_{\varepsilon}$ be a point such
that dist$(x, x_0)=$ dist$(x,\partial\Omega_{\varepsilon})$ and consider $%
\bar{ u}^{\varepsilon}(x)=\bar{ u}_{\varepsilon}(P_{\varepsilon}+\varepsilon
x)$ in $B_2(x_0)$, the ball of radius $2$ centered at $x_0$, then from (\ref%
{equation2}) it follows that $\bar{u}^{\varepsilon}$ satisfies 
\begin{equation}
\limfunc{div}\left( \tilde{A}(P_{\varepsilon}+\varepsilon x,\nabla\bar {u}%
^{\varepsilon})\right) +B(P_{\varepsilon}+\varepsilon x,\bar{u}%
^{\varepsilon})=0\quad\text{in}\quad B_2(x_0)  \label{add1equation2}
\end{equation}
in the weak sense. Then applying an $C^{1,\alpha}$ estimate (see \cite{Tol},
for example) again yields as above that there exists two constants $C>0$ and 
$\alpha_*\in (0,1)$ which are independent of $\varepsilon$ such that 
\begin{equation}
\begin{aligned} \|\hat{u}^{\varepsilon}\|_{C^{1,\alpha_*}(B_{1}(x_0))}&\leq
C\left(\|\hat {u}^{\varepsilon}\|_{L^{\infty}(B_2(x_0))}
+\|B(P_{\varepsilon} +\varepsilon x,
\hat{u}^{\varepsilon})\|^{\frac{1}{m-1}}_{L^{\infty}(B_2(x_0))}\right)\\
\nonumber&\leq c_3^*\exp\{-c_4^*|x-P_{\varepsilon}|\} \end{aligned}
\label{addexp11}
\end{equation}
by adjusting $c_3^*$ and $c_4^*$ if it is necessary. Especially we have 
\begin{equation}
|\nabla u^{\varepsilon}(x)|\leq c_3^*\exp\{-c_4^*|x-P_{\varepsilon}|\},
\label{addexp21}
\end{equation}
Thus combining (\ref{addexp1}) and (\ref{addexp21}) together and scaling
back we have for $x\in \Omega$ 
\begin{equation}
|\nabla u_{\varepsilon}(x)|\leq c_3^*\varepsilon^{-1}\exp\{-c_4^*\frac{%
|x-P_{\varepsilon}|}{\varepsilon}\}.  \notag
\end{equation}
Proof of Lemma \ref{eqSom.3} is completed by letting $c_3=\max\{\tilde{c_3},
c_3^*\}$ and $c_4=\min\{\tilde{c_4}, c_4^*\}$.
\end{proof}

\begin{remark}
\label{RemSom.2} Our proof of the Lemma \ref{eqSom.3} with necessary minor
modifications also works well for elliptic systems.
\end{remark}

Next we present a lemma related to extensions of $u_{\varepsilon}$.

\begin{lemma}
\label{LemSom.4}There exists a $C^{1}$-extension $\tilde{u}_{\varepsilon}$
of $u_{\varepsilon}$ which has compact support in $\mathbb{R}^{N}$ and
satisfies

\begin{enumerate}
\item \renewcommand{\theenumi}{\roman{enumi}}

\item \label{LemSom.4(1)}$\displaystyle\left\Vert \tilde{u}_{\varepsilon
}\right\Vert _{W^{1,m}\left( \mathbb{R}^{N}\right) }\leq c_{5}\left\Vert
u_{\varepsilon}\right\Vert _{W^{1,m}\left( \Omega\right) }$ and $%
\displaystyle\left\Vert \tilde{u}_{\varepsilon}\right\Vert _{C^{1}\left( 
\mathbb{R}^{N}\right) }\leq c_{5}\left\Vert u_{\varepsilon}\right\Vert
_{C^{1}\left( \bar{\Omega}\right) }$,

\item \label{LemSom.4(2)}$\tilde{u}_{\varepsilon}$ also has the
exponential-decay property as stated in Lemma \textup{\ref{LemSom.2}}, i.e.,
there exists an absolute constant $\lambda\geq1$ such that%
\begin{equation}
\begin{aligned} 0\leq\tilde{u}_{\varepsilon}&\leq c_{3}\lambda\exp\left\{
-\frac{c_{4}}{\lambda} \frac{\left\vert x-P_{\varepsilon}\right\vert
}{\varepsilon }\right\},\\ \left\vert\nabla
\tilde{u}_{\varepsilon}(x)\right\vert &\leq c_3\lambda
\varepsilon^{-1}\exp\{-\frac{c_4}{\lambda}\frac{|x-P_{\varepsilon}|}{
\varepsilon}\}. \end{aligned}  \label{eqSom.22}
\end{equation}
\end{enumerate}

\noindent\setcounter{saveenumi}{\value{enumi}}and

\begin{enumerate}
\item \setcounter{enumi}{\value{saveenumi}}\renewcommand{\theenumi}{%
\roman{enumi}}

\item \label{LemSom.4(3)}there exists a positive constant $\tilde\delta
=\tilde\delta\left( \Omega\right) $ such that for any $P\in\partial\Omega$, $%
\tilde{u}_{\varepsilon}|_{B_{\tilde\delta}\left( P\right) \setminus\Omega}$
is the reflection of $u_{\varepsilon}$ through $\partial\Omega$.
\end{enumerate}
\end{lemma}

\begin{proof}
Let $\tilde d=d\left( \partial\Omega, \partial\Omega^{\tilde\gamma}\right) $
and $0\leq\varrho(x)\leq1$ be a smooth cut-off function such that $%
\varrho(x)\equiv1$ for $x\in\{x\in\mathbb{R}^{N}, d(x,\Omega )\leq\frac{%
\tilde d}{2}\}$ and $\varrho(x)\equiv0$ for $x\in\mathbb{R}%
^{N}\setminus\left( \overline{\Omega}\bigcup\Omega_{O}\right) .$ Then $%
\tilde u_{\varepsilon}=\varrho\bar u_{\varepsilon}$ satisfies (\ref%
{LemSom.4(1)}), (\ref{LemSom.4(2)}) and (\ref{LemSom.4(3)}) automatically.
The proof of this lemma is completed.
\end{proof}

Similar to energy density introduced in \cite{DeFe99}, we define the energy
density associated with (\ref{eqInt.1}) as follows: 
\begin{equation*}
E\left( w,y^{\prime}\right) =\left[ \frac{1}{m}\left( \left\vert \nabla
w\right\vert ^{m}+w^{m}\right) -F\left( w\right) \right] \left(
y^{\prime},0\right) \text{\quad for }y^{\prime}\in\mathbb{R}^{N-1}.
\end{equation*}
Then we have the following lemma.

\begin{lemma}
\label{LemSom.3}Let $G$ be a $C^{2}$ function in a neighborhood of the
origin of $\mathbb{R}^{N-1}$. Then%
\begin{equation*}
\sum_{i,j=1}^{N-1}\int_{\mathbb{R}^{N-1}}G_{ij}\left( 0\right)
y_{i}y_{j}E\left( w,y^{\prime}\right) \,dy^{\prime}=2\Delta G\left( 0\right)
\gamma,
\end{equation*}
where $\gamma$ is the constant defined in \textup{(\ref{eqInt.13})}, and $%
y^{\prime}=\left( y_{1},\dots,y_{N-1}\right) $, and 
\begin{equation*}
G_{ij}\left( 0\right) =\frac{\partial^{2}G}{\partial y_{i}\partial y_{j}}%
\left( 0\right) .
\end{equation*}
\end{lemma}

\begin{proof}
In Lemma 2.4 of \cite{LiZh04c}, we showed that%
\begin{equation}
\gamma=\frac{1}{2}\int_{\mathbb{R}_{+}^{N}}\left( \frac{1}{m}\left(
\left\vert \nabla w\right\vert ^{m}\right) +w^{m}-F(w) \right) z_{N}\,dz.
\label{eqSom.19}
\end{equation}
Next we introduce the polar coordinates%
\begin{equation*}
\left\{ \begin{aligned}
z_{1}&=r\sin\theta_{N-1}\sin\theta_{N-2}\cdots\sin\theta_{2}\sin
\theta_{1},\\
z_{2}&=r\sin\theta_{N-1}\sin\theta_{N-2}\cdots\sin\theta_{2}\cos\theta_{1},%
\\ z_{3}&=r\sin\theta_{N-1}\sin\theta_{N-2}\cdots\cos\theta _{2},\\
&\vdots\;,\\ z_{N}&=r\cos\theta_{N-1}, \end{aligned} \right.
\end{equation*}
and notice that%
\begin{multline*}
\mathbb{R}_{+}^{N}=\left\{ \,\left( r,\theta_{1},\dots,\theta_{N-1}\right)
\mid r>0,\;0\leq\theta_{1}<2\pi ,\vphantom{\;0\leq\theta_{j}<\pi\text{ for
}j=2,\dots,N-2\text{, and }0\leq\theta_{N-1}<\frac{\pi}{2}\,}\right. \\
\left. \vphantom{\,\left( r,\theta_{1},\dots,\theta_{N-1}\right) \mid
r>0,\;0\leq\theta_{1}<2\pi,\;}0\leq \theta_{j}<\pi\text{ for }j=2,\dots,N-2%
\text{, and }0\leq\theta_{N-1}<\frac{\pi}{2}\,\right\}
\end{multline*}
and that%
\begin{equation*}
dz=r^{N-1}\sin\theta_{2}\sin^{2}\theta_{3}\cdots\sin^{N-2}\theta
_{N-1}\,dr\,d\theta_{1}\cdots d\theta_{N-1}.
\end{equation*}
After elementary computations one obtains%
\begin{equation}
\gamma=\frac{1}{2}\int_{0}^{\infty}\left( \frac{1}{m}\left( \left\vert
w^{\prime}\left( r\right) \right\vert ^{m}+w^{m}\left( r\right) \right)
-F\left( w\left( r\right) \right) \right) r^{N}\,dr\cdot\omega_{N-2},
\label{eqSom.20}
\end{equation}
where $\omega_{N-2}$ is the volume of the unit ball in $\mathbb{R}^{N-2}$.
Here we used the fact that $w$ is radially symmetric.

Using the radial symmetry of $w$ again, we obtain%
\begin{align}
& \sum_{i,j=1}^{N-1}\int_{\mathbb{R}^{N-1}}G_{ij}\left( 0\right)
y_{i}y_{j}E\left( w,y^{\prime}\right) \,dy^{\prime}  \label{eqSom.21} \\
& \qquad=\sum_{i=1}^{N-1}\int_{\mathbb{R}^{N-1}}G_{ii}\left( 0\right)
y_{i}^{2}E\left( w,y^{\prime}\right) \,dy^{\prime}  \notag \\
& \qquad=\sum_{i=1}^{N}G_{ii}\left( 0\right) \cdot\frac{1}{N-1}\int_{\mathbb{%
R}^{N-1}}\left\vert y^{\prime}\right\vert ^{2}E\left( w,y^{\prime}\right)
\,dy^{\prime}  \notag \\
& \qquad=\Delta G\left( 0\right) \cdot\int_{0}^{\infty}E\left( w,r\right)
r^{N}\,dr\cdot\omega_{N-2},  \notag
\end{align}
where $E\left( w,r\right) =\left( 1/m\right) \left( \left\vert w^{\prime
}\left( r\right) \right\vert ^{m}+w^{m}\left( r\right) \right) -F\left(
w\left( r\right) \right) \rule{0pt}{18pt}$. Comparing (\ref{eqSom.20}) and (%
\ref{eqSom.21}) yields%
\begin{equation*}
\sum_{i,j=1}^{N-1}\int_{\mathbb{R}^{N-1}}G_{ij}\left( 0\right)
y_{i}y_{j}E\left( w,y^{\prime}\right) \,dy^{\prime}=2\Delta G\left( 0\right)
\gamma.
\end{equation*}
The proof of Lemma \ref{LemSom.3} is completed.
\end{proof}

\section{\label{ProB}Proof of Theorem \protect\ref{ThmB}}

With the help of the lemmas in Section \ref{Som}, now we can give the proof
of Theorem \ref{ThmB}.

\begin{proof}[Proof of Theorem \textup{\protect\ref{ThmB}}]
Since as $\varepsilon\rightarrow 0^+$, $P_{\varepsilon}\rightarrow\partial%
\Omega$ at the rate of $o(\varepsilon)$, it follows that $%
d(P_{\varepsilon},\tilde P_{\varepsilon})/\varepsilon\rightarrow 0$, where $%
\tilde P_{\varepsilon}\in\partial\Omega$ is the closest point on $%
\partial\Omega$ to $P_{\varepsilon}$. then by passing to a sequence, $\tilde
P_{\varepsilon}\rightarrow\bar{P}\in \partial\Omega$. After an $\varepsilon$%
-dependent rotation and translation, we may assume that $\tilde
P_{\varepsilon}$ is at the origin and $\Omega$ can be described in a fixed
cubic neighborhood $V$ of $\bar{P}$ as the set%
\begin{equation*}
\left\{ \,\left( x^{\prime},x_{N}\right) \mid x_{N}>\psi_{\varepsilon
}\left( x^{\prime}\right) \,\right\} \text{\qquad with }x^{\prime}=\left(
x_{1},\dots,x_{N-1}\right) ,
\end{equation*}
where $\psi_{\varepsilon}$ is smooth, $\psi_{\varepsilon}\left( 0\right) =0$%
, $\nabla\psi_{\varepsilon}\left( 0\right) =0$. Furthermore, we may assume
that $\psi_{\varepsilon}$ converges locally in the $C^{2}$ sense to $\psi$,
a corresponding parametrization at $\bar{P}$. Note that since $\tilde
P_{\varepsilon}$ is the origin, so we have $P_{\varepsilon}/\varepsilon%
\rightarrow 0$ as $\varepsilon\rightarrow 0^+.$ Thus we have $\tilde
u^{\varepsilon}(x)=\tilde u_{\varepsilon}(\varepsilon x)=$ $\tilde
u_{\varepsilon}\left(\varepsilon\left(x-\frac{P_{\varepsilon}}{\varepsilon}%
\right)+P_{\varepsilon}\right)\rightarrow w(x)$ in $C^{1}_{loc}\left(\mathbb{%
R}^N\right)$ as $\varepsilon\rightarrow 0^+$. From the characterization of $%
c_{\varepsilon}=J_{\varepsilon}\left( u_{\varepsilon }\right) $ in Section %
\ref{Int}, we have%
\begin{equation*}
\varepsilon^{-N}J_{\varepsilon}\left( u_{\varepsilon}\right) \geq
\varepsilon^{-N}J_{\varepsilon}\left( tu_{\varepsilon}\right) =I_{\Omega
_{\varepsilon}}\left( tu^{\varepsilon}\right)
\end{equation*}
for all $t>0$. Hereinafter%
\begin{equation*}
I_{\Omega_{\ast}}\left( v\right) =\frac{1}{m}\int_{\Omega_{\ast}}\left(
\left\vert \nabla v\right\vert ^{m}+\left\vert v\right\vert ^{m}\right)
\,dx-\int_{\Omega_{\ast}}F\left( v\right) \,dx.
\end{equation*}
Then%
\begin{align}
I_{\Omega_{\varepsilon}}\left( tu^{\varepsilon}\right) & =I_{\Omega
_{\varepsilon}}\left( t\tilde{u}^{\varepsilon}\right) \geq I_{\mathbb{R}%
_{+}^{N}\cap V_{\varepsilon}}\left( t\tilde{u}^{\varepsilon}\right)
+I_{\left( \Omega_{\varepsilon}\cap V_{\varepsilon}\right) \setminus \mathbb{%
R}_{+}^{N}}\left( t\tilde{u}^{\varepsilon}\right) -I_{\left( \mathbb{R}%
_{+}^{N}\cap V_{\varepsilon}\right) \setminus\Omega_{\varepsilon}}\left( t%
\tilde{u}^{\varepsilon}\right)  \label{eqProB.1} \\
& =\mathrm{I}+\mathrm{II}-\mathrm{III},  \notag
\end{align}
with $V_{\varepsilon}=\frac{1}{\varepsilon}V.$ Let us choose $%
t=t_{\varepsilon}$ so that $I_{\mathbb{R}_{+}^{N}}\left( t\tilde{u}%
^{\varepsilon}\right) $ maximizes in $t$. Then from the definition of $%
C_{\ast}$ in (\ref{eqInt.10}), equality (\ref{eqInt.11}) and Lemma \ \ref%
{LemSom.4} it follows that 
\begin{equation*}
\mathrm{I}=I_{\mathbb{R}_{+}^{N}\cap V_{\varepsilon}}\left( t_{\varepsilon }%
\tilde{u}^{\varepsilon}\right) \geq\frac{c_{\ast}}{2}+O\left(
e^{-c_{6}/\varepsilon}\right)
\end{equation*}
for some constant $c_{6}>0$ independent of $\varepsilon$. Next we give an
estimate of $t_{\varepsilon}.$

\begin{lemma}
\label{LemSom.5}There is a unique $t_{\varepsilon}\in\left( 0,\infty\right) $
such that%
\begin{multline*}
\frac{1}{m}\int_{\mathbb{R}_{+}^{N}}t_{\varepsilon}^{m}\left( \left\vert
\nabla\tilde{u}^{\varepsilon}\right\vert ^{m}+(\tilde{u}^{\varepsilon})^{m}%
\right) \,dx-\int_{\mathbb{R}_{+}^{N}}F\left( t_{\varepsilon}\tilde {u}%
^{\varepsilon}\right) \,dx \\
=\sup_{t\geq0}\left[ \frac{1}{m}\int_{\mathbb{R}_{+}^{N}}t^{m}\left(
\left\vert \nabla\tilde{u}^{\varepsilon}\right\vert ^{m}+(\tilde {u}%
^{\varepsilon})^{m}\right) \,dx-\int_{\mathbb{R}_{+}^{N}}F\left( t\tilde{u}%
^{\varepsilon}\right) \,dx\right] ,
\end{multline*}
and moreover%
\begin{equation}
t_{\varepsilon}=1+o\left( 1\right) \text{\quad as }\varepsilon
\rightarrow0^{+}.  \label{eqSom.23}
\end{equation}
\end{lemma}

\begin{proof}
Under assumption (\ref{hypothesis4}), the existence and uniqueness of $%
t_{\varepsilon}$ can be proved similarly to the proof of Lemma 2.1 of \cite%
{LiZh04c}. Here we only need show (\ref{eqSom.23}). Let 
\begin{equation}
h_{\varepsilon}\left( t\right) =\frac{t^{m}}{m}\int_{\mathbb{R}%
_{+}^{N}}\left( \left\vert \nabla\tilde{u}^{\varepsilon}\right\vert ^{m}+(%
\tilde {u}^{\varepsilon})^{m}\right) \,dx-\int_{\mathbb{R}_{+}^{N}}F\left( t%
\tilde{u}^{\varepsilon}\right) \,dx.  \label{qqq}
\end{equation}
Then%
\begin{equation}
\begin{aligned} h_{\varepsilon}^{\prime}\left( t\right) & =
t^{m-1}\int_{\mathbb{R}_{+}^{N}}\left( \left\vert \nabla
\tilde{u}^{\varepsilon}\right\vert ^{m}
+(\tilde{u}^{\varepsilon})^{m}\right) \,dx
-\int_{\mathbb{R}_{+}^{N}}\tilde{u}^{\varepsilon}f\left(
t\tilde{u}^{\varepsilon}\right) \,dx\\ &
=t^{m-1}\int_{\mathbb{R}_{+}^{N}}\left( \left\vert \nabla w\right\vert ^{m}
+w^{m}\right) \,dx -\int_{\mathbb{R}_{+}^{N}}wf\left( tw\right) \,dx+o(1),
\end{aligned}  \label{eqSom.24}
\end{equation}
here we have used the exponential decay of $\tilde{u}_{\varepsilon}$ in
Lemma \ref{LemSom.4}, exponential decay of $w$ and $\tilde{u}^{\varepsilon
}\rightarrow w$ in $C_{\limfunc{loc}}^{1}\left( \mathbb{R}^{N}\right) $ as $%
\varepsilon\rightarrow0^{+}$. Moreover the term $o\left( 1\right)
\rightarrow0$ uniformly in $t$ on each compact interval as $\varepsilon
\rightarrow0^{+}$. (\ref{qqq}) tells us $h_{\varepsilon}(1)=\frac{1}{2}%
c_{\ast}+o(1)$, which yields that $t_{\varepsilon}$ is bounded and away from 
$0$. Also from (\ref{eqSom.24}) it follows that 
\begin{equation}
\begin{aligned} h_{\varepsilon}^{\prime}\left(t\right)&=t^{m-1}
\int_{\mathbb{R}_+^{N}}wf\left(w\right)\,dx
-\int_{\mathbb{R}_+^{N}}w\,f\left(tw\right) \,dx +o\left( 1\right)\\
&=t^{m-1}\int_{\mathbb{R}_+^N}w^m\left(\frac{f\left(w\right)}{w^{m-1}}-
\frac{f\left(tw\right)}{\left(tw\right)^{m-1}}\right)\,dx+o\left( 1\right).
\end{aligned}  \label{eqSom.26}
\end{equation}
Therefore at $t=t_{\varepsilon}$ we have 
\begin{equation}
\int_{\mathbb{R}_{+}^{N}}w^{m}\left( \frac{f\left( w\right) }{w^{m-1}}-\frac{%
f\left( t_{\varepsilon}w\right) }{\left( t_{\varepsilon}w\right) ^{m-1}}%
\right) \,dx=o\left( 1\right) .  \label{eqSom.27}
\end{equation}
Since $f(t)/t^{m-1}$ is strictly increasing (see (\ref{hypothesis4})) it
follows from (\ref{eqSom.27}) that $t_{\varepsilon}=1+o\left( 1\right) .$
The proof of Lemma \ref{LemSom.5} is completed.
\end{proof}

\textit{Proof of Theorem \textup{\ref{ThmB}} continued}. Using again the
exponential decay of $\ u_{\varepsilon}$ in Lemma \ref{LemSom.2} and the
expansion of $t_{\varepsilon}$ in Lemma \ref{LemSom.5}, we obtain%
\begin{align}
-\mathrm{II} & =-\int_{\left( \mathbb{R}^{N-1}\times\left\{ 0\right\}
\right) \cap V_{\varepsilon}}dy^{\prime}  \label{eqProB.2} \\
& \qquad\quad{}\cdot\int_{\frac{\left( \psi_{\varepsilon}\left( \varepsilon
y^{\prime}\right) \right) _{-}}{\varepsilon}}^{0}\left[ \frac{1}{m}%
t_{\varepsilon}^{m}\left( \left\vert \nabla\tilde{u}^{\varepsilon
}\right\vert ^{m}+\left( \tilde{u}^{\varepsilon}\right) ^{m}\right) -F\left(
t_{\varepsilon}\tilde{u}^{\varepsilon}\right) \right] \left(
y^{\prime},y_{N}\right) \,dy_{N}  \notag \\
& =-\left( 1+o\left( 1\right) \right) \int_{\left( \mathbb{R}%
^{N-1}\times\left\{ 0\right\} \right) \cap\left( \Omega_{\varepsilon}\cap
V_{\varepsilon}\right) }dy^{\prime}  \notag \\
& \qquad\quad{}\cdot\int_{\frac{\left( \psi_{\varepsilon}\left( \varepsilon
y^{\prime}\right) \right) _{-}}{\varepsilon}}^{0}\left[ \frac{1}{m}\left(
\left\vert \nabla u^{\varepsilon}\right\vert ^{m}+\left( u^{\varepsilon
}\right) ^{m}\right) -F\left( u^{\varepsilon}\right) \right] \left(
y^{\prime},y_{N}\right) \,dy_{N}.  \notag
\end{align}
Similarly, 
\begin{multline}
\mathrm{III}=\left( 1+o\left( 1\right) \right)
\int_{V_{\varepsilon}\cap\left( \mathbb{R}^{N-1}\times\left\{ 0\right\}
\right) }dy^{\prime} \\
{}\cdot\int_{0}^{\frac{\left( \psi_{\varepsilon}\left( \varepsilon
y^{\prime}\right) \right) _{+}}{\varepsilon}}\left[ \frac{1}{m}\left(
\left\vert \nabla\tilde{u}^{\varepsilon}\right\vert ^{m}+\left( \tilde {u}%
^{\varepsilon}\right) ^{m}\right) -F\left( \tilde{u}^{\varepsilon }\right) %
\right] \left( y^{\prime},y_{N}\right) \,dy_{N}.  \label{eqProB.3}
\end{multline}
In \ above $a_{+}=\max\{a,0\},$ $a_{-}=\min\{a,0\}.$\ Since $\psi
_{\varepsilon}\left( 0\right) =0$, $\nabla\psi_{\varepsilon}\left( 0\right)
=0$ and $\psi_{\varepsilon}$ converges in the $C^{2}$ local sense to $\psi$,
and $\tilde{u}^{\varepsilon}\rightarrow w$ in the $C^{1}$ local sense in $%
\mathbb{R}^{N}$ with uniform exponential decay with respect to $\varepsilon$%
, it follows from the dominated convergence theorem that 
\begin{align*}
& \lim_{\varepsilon\rightarrow0^{+}}\frac{1}{\varepsilon}\left( -\mathrm{II}+%
\mathrm{III}\right) \\
& \qquad=\frac{1}{2}\sum_{i,j=1}^{N=1}\int_{\mathbb{R}^{N-1}}\psi_{ij}\left(
0\right) y_{i}y_{j}\left( \frac{1}{m}\left( \left\vert \nabla w\right\vert
^{m}+w^{m}\right) -F\left( w\right) \right) \left( y^{\prime},0\right)
\,dy^{\prime} \\
& \qquad=\Delta\psi\left( 0\right) \gamma=\left( N-1\right) H\left( \bar{P}%
\right) \gamma\text{\qquad(by Lemma \ref{LemSom.3}).}
\end{align*}
Thus we have%
\begin{equation*}
c_{\varepsilon}\geq\varepsilon^{N}\left\{ \frac{1}{2}c_{\ast}-\left(
N-1\right) H\left( \bar{P}\right) \gamma\varepsilon+o\left( \varepsilon
\right) \right\} .
\end{equation*}
But (\ref{eqInt.12}) in Theorem \ref{ThmA} tells us%
\begin{equation*}
c_{\varepsilon}\leq\varepsilon^{N}\left\{ \frac{1}{2}c_{\ast}-\left(
N-1\right) \max_{P\in\partial\Omega}H\left( P\right) \gamma\varepsilon
+o\left( \varepsilon\right) \right\} .
\end{equation*}
Therefore we get

\begin{enumerate}
\item \renewcommand{\theenumi}{\roman{enumi}}

\item \label{ThmBproof(1)} $\displaystyle H\left( \bar{P}\right) =\max
_{P\in\partial\Omega}H\left( P\right) $, which is (\ref{ThmB(1)}) of Theorem %
\ref{ThmB},
\end{enumerate}

\noindent\setcounter{saveenumi}{\value{enumi}}and

\begin{enumerate}
\item \setcounter{enumi}{\value{saveenumi}}\renewcommand{\theenumi}{%
\roman{enumi}}

\item \label{ThmBproof(2)} $\displaystyle
c_{\varepsilon}=\varepsilon^{N}\left\{ \frac{1}{2}c_{\ast}-\left( N-1\right)
H\left( \bar{P}\right) \gamma\varepsilon+o\left( \varepsilon \right)
\right\} $ as $\varepsilon\rightarrow0^{+}$,
\end{enumerate}

\noindent which is part (\ref{ThmB(2)}) of Theorem \ref{ThmB}. The proof of
Theorem \ref{ThmB} is completed.
\end{proof}

\textbf{Acknowledgement.} The authors want to give their thanks to anonymous
referee for some helpful comments.

\providecommand{\bysame}{\leavevmode\hbox to3em{\hrulefill}\thinspace} %
\providecommand{\MR}{\relax\ifhmode\unskip\space\fi MR } 
\providecommand{\MRhref}[2]{  \href{http://www.ams.org/mathscinet-getitem?mr=#1}{#2}
} \providecommand{\href}[2]{#2}

\end{document}